\documentclass[a4paper,10pt]{article}

\usepackage{amsfonts,amssymb,amsmath}

\numberwithin{equation}{section}
\usepackage[labelsep=endash]{caption}

\usepackage[normalem]{ulem}

\usepackage{mdwlist}
\usepackage{tikz-cd}

\newtheorem{theo}{Theorem}[section]
\newtheorem{coro}[theo]{Corollary}
\newtheorem{lemm}[theo]{Lemma}
\newtheorem{prop}[theo]{Proposition}
\newtheorem{defi}[theo]{Definition}
\newtheorem{rema}[theo]{Remark}

\newtheorem{exams}[theo]{Examples}

\newenvironment{proof}{\noindent \textbf{{Proof.}} \sf}

\def\qed{\hfill $\diamond$ \bigskip}

\def\lim{\mathop{\rm lim}\nolimits}

\def\Ext{\mathsf{Ext}}
\def\Hom{\mathsf{Hom}}
\def\Tor{\mathsf{Tor}}

\def\Ker{\mathsf{Ker}}
\def\ker{\mathop{\rm Ker}\nolimits}
\def\Coker{\mathsf{Coker}}

\def\Im{\mathsf{Im}}

\begin{document}
\sf

\title{Split bounded extension algebras and Han's conjecture}
\author{Claude Cibils,  Marcelo Lanzilotta, Eduardo N. Marcos,\\and Andrea Solotar
\thanks{\footnotesize This work has been supported by the projects  UBACYT 20020130100533BA, PIP-CONICET 11220150100483CO, USP-COFECUB.
The third mentioned author was supported by the thematic project of FAPESP 2014/09310-5 and acknowledges support from the "Brazilian-French Network in Mathematics". The fourth mentioned author is a research member of CONICET (Argentina) and a Senior Associate at ICTP.}}

\date{}
\maketitle
\begin{abstract}
A main purpose of this paper is to prove that the class of finite dimensional algebras which verify Han's conjecture is closed under split bounded extensions.

\end{abstract}



\section{\sf Introduction}

Given a finite dimensional algebra $A$ over an algebraically closed field $k$, Han's conjecture relates two homological invariants associated to $A$:
its global dimension and its Hochschild homology. In the commutative case -- non necessarily finite dimensional but finitely generated -- the finiteness of the global dimension is equivalent to the fact that $A$ is geometrically regular \cite{AUSLANDERBUCHSBAUM,Se}. In general we are going to say that an algebra with finite global dimension is {\em smooth}.

On the other hand, we consider Hochschild homology of $A$. Let $A^e=A\otimes A^{op}$ be the enveloping algebra. Let us recall that given an $A$-bimodule $X$ -- or equivalently a left  or right
$A^e$-module --, the Hochschild homology of $A$ with coefficients in $X$ is
$H_*(A,X)=\Tor^{A^e}_*(A,X)$; it is functorial in both variables.

Han's conjecture \cite{HAN} states that for $A$ finite dimensional, $A$ is smooth if and only if $H_n(A,A)=0$ for $n>>0$. The direct implication is true.

Next we recall some previous results. Well before being formulated, Han's conjecture has been proved for commutative algebras which are finitely generated, which encompasses finite dimensional commutative algebras, see \cite{BACH,AVRAMOVVIGUE}. Y. Han proved the conjecture for monomial algebras in \cite{HAN}. P.A. Bergh and D. Madsen have shown that it holds in characteristic zero for graded finite dimensional local algebras, Koszul algebras and graded cellular algebras  \cite{BERGHMADSEN2009}. They have also obtained a confirmation of Han's conjecture in  \cite{BERGHMADSEN2017} for trivial extensions of several sorts of algebras, by proving that their Hochschild homology is non zero in large enough degrees. P.A. Bergh and K. Erdmann  proved in
\cite{BERGHERDMANN} that quantum complete intersections - at a non-root of unity - satisfy Han's conjecture, as well as A. Solotar and
M. Vigu\'{e}-Poirrier \cite{SOLOTARVIGUE} for a generalisation of quantum complete intersections and for a family of algebras which are in some
sense opposite to these last ones.
Later,  A. Solotar, M. Su\'{a}rez-Alvarez and Q. Vivas proved in \cite{SOLOTARSUAREZVIVAS} Han's conjecture for quantum generalized Weyl algebras
(out of a few exceptional cases). In \cite{CIBILSREDONDOSOLOTAR} null-square projective algebras extensions were considered, the present paper goes further in this direction.

Concerning the commutative case, it is worth to mention that in characteristic zero, in positive degrees $HH_n(A)$ has a decomposition,
 called Hodge decomposition, as a direct sum of subspaces, see for example \cite{GS, Ronco, Vigue}. One of them is the $n$-th exterior power of the $A$-module of K\"ahler differentials, $\Omega^n_{A|k}$ and another one is $D_n(A|k)$, the Andr\'e-Quillen homology of the commutative $k$-algebra $A$. When $A$ is smooth, in positive degrees $HH_n(A)= \Omega^n_{A|k}$ and the other summands annihilate.
In fact, the main condition for smoothness is that $D_n(A|k)=0$ for positive $n$ \cite{Iyengar}, and the Jacobi-Zariski long exact sequence
for Andr\'e-Quillen homology relating $D_n(A|k)$, $D_n(A|B)$ and $D_n(B|k)$ for any extension of algebras $k\subseteq B \subseteq A$ plays an
important role.

In the non commutative setting Andr\'e-Quillen homology does not exist, but A.~Kaygun has proved recently in \cite{KAYGUN,KAYGUNe}
the existence of a
Jacobi-Zariski long exact sequence starting in degree one for Hochschild homology for any extension of $k$-algebras
$B\subseteq A$, such that $A$ is $B$-flat. It relates the
ordinary Hochschild homologies of $A$ and $B$ with the relative Hochschild homology of $A$ with respect to $B$.
In this paper, with different hypotheses we also obtain a long exact sequence of
Jacobi-Zariski type for large enough degrees.

We consider split extension algebras in relation with Han's conjecture. By definition, a split extension algebra over a field $k$ is a $k$-algebra of the form $A=B\oplus M$, where $B$ is a subalgebra of $A$ and $M$ is a two-sided ideal of $A$.
As a consequence of our work, we prove that in some cases, adding or deleting arrows to a quiver -- even adding or deleting certain relations --
does not change the situation with respect to Han's conjecture,  see also \cite{CIBILSLANZILOTTAMARCOSSOLOTAR}. Indeed, these processes are special cases of split extension algebras, see Example \ref{examples split}.\ref{all}. In a subsequent work, conditions will be given for these operations to fit within the framework of the theory we provide in this paper.
\normalsize

Next we describe the contents of this article. In Section \ref{section rrbr}, in order to compute the relative Hochschild (co)homology introduced by G. Hochschild in \cite{HOCHSCHILD1956}, we construct a reduced relative bar resolution of a split extension algebra.
We use it particularly when $M$ is $B$-tensor nilpotent, that is if there exists $n$ such that $M^{\otimes_B n}=0$.

In Section \ref{nearly} we obtain a Jacobi-Zariski long exact sequence in the following situation.
A $B$-bimodule $M$ is called bounded if $M$ is $B$-tensor nilpotent, of finite projective dimension as $B$-bimodule and projective either as
left or as right $B$-module. A split bounded extension algebra is a split extension $A=B\oplus M$ where $M$ is bounded.
For these algebras we obtain a Jacobi-Zariski long exact sequence in large enough degrees. We set up techniques based on nearly exact sequences of complexes, see Definition
\ref{sequence nearly exact}.
Actually the relative  resolution of Section \ref{section rrbr} provides a nearly exact sequence, which in turn gives the required Jacobi-Zariski long exact sequence in large enough degrees.

In Section \ref{Han} we prove our main result: the class $\mathcal{H}$ of finite dimensional algebras  which verify Han's conjecture is closed under split bounded extensions. More precisely if $A=B\oplus M$ is such an extension, then $A\in\mathcal{H}$ if and only if $B\in\mathcal{H}$. We point out that this result does not depend on the associative structure of $M$, but on properties of its $B$-bimodule structure, see Definitions \ref{bounded} and \ref{split bounded extension}.The proofs make use of the Jacobi-Zariski long exact sequence, and of the reduced relative bar resolution.

\section{\sf A reduced relative bar resolution for split extension algebras}\label{section rrbr}

Let $B\subset A$ be an extension of algebras over a field $k$. In this context G. Hochschild introduced in \cite{HOCHSCHILD1956}
relative homological algebra, which   corresponds to consider the exact category of $A$-modules with respect to $B$-split short exact sequences, see \cite{QUILLEN,BUHLER}.  More precisely, an\emph{ induced module} is an $A$-module of the form $A\otimes_BM$, where $M$ is a left $B$-module.  An $A$-module $P$ is  \emph{relative projective} if
any $A$-morphism $X\to P$ which has a $B$-section has an $A$-section.
Equivalently, $P$ is relative projective if it is an $A$-direct summand of an induced module.
There are enough relative projectives since for any $A$-module $X$ the canonical $A$-map $A\otimes_B X\to X$ has a $B$-section.
Of course if $B=k$ we recover the ordinary definition, and if $B=A$ all modules are relative projective.

A relative  projective resolution of an $A$-module $X$ is a sequence
$$\cdots\stackrel{d}{\to}P_2\stackrel{d}{\to}P_1\stackrel{d}{\to}P_0\to X \to 0$$
where each $P_i$ is a relative projective $A$-module, the $d$'s are $A$-morphisms, $d^2=0$ and there exists a $B$-contracting homotopy, see \cite[p. 250]{HOCHSCHILD1956}.

Two relative projective resolutions of $X$ are homotopic and the functor $A\otimes_B -$ is exact,
so that for $X$ and $Y$ respectively right and  left $A$-modules, the functor $\Tor_*^{A|B}(X,Y)$ is well defined. For $X$ and $Y$ left $A$-modules, the functor $\Ext^*_{A|B}(X,Y)$ is well defined.

Consider the extension of enveloping algebras $B^e\subset A^e$. For $X$ an $A$-bimodule, the relative Hochschild homology and cohomology vector spaces are  defined in \cite{HOCHSCHILD1956} respectively as follows:
$$H_*(A|B,X)= \Tor_*^{A^e|B^e}(X,A)\mbox{\ \ and\ \ }H^*(A|B,X) = \Ext^*_{A^e|B^e}(A,X).$$

Observe that in \cite{HOCHSCHILD1956} those vector spaces are defined with respect to the extension $B\otimes A^{op}\subset A^e$. This turns out to be equivalent since the relative canonical resolution of $A$ is relative projective in both situations, and the canonical contracting homotopies agree.

Being derived functors, they can be computed using an arbitrary relative projective resolution. In particular these vector spaces are the homology and the cohomology of the following chains and cochains complexes $C_*(A|B,X)$ and $C^*(A|B,X)$:
$$\cdots \stackrel{b}{\to} X\otimes_{B^e} A^{\otimes_Bn}\stackrel{b}{\to}\cdots \stackrel{b}{\to} X\otimes_{B^e}A \stackrel{b}{\to} X_B\to 0,$$
$$0\to X^B\stackrel{b}{\to}\Hom_{B^e}(A,X)\stackrel{b}{\to}\cdots \stackrel{b}{\to}\Hom_{B^e}(A^{\otimes_Bn},X)\stackrel{b}{\to}\cdots$$
where $$X_B= X\otimes_{B^e}B= X/\langle bx-xb\rangle = H_0(B,X),$$
$$X^B=\Hom_{B^e}(B,X)=\{x\in X\mid bx=xb \mbox{ for all } b\in B\}= H^0(B,X),$$
and where the formulas for the boundaries and coboundaries are the ordinary ones.

\begin{defi}
An extension of algebras $B\subset A$ is \emph{split} if there is a morphism of algebras $\pi: A\to B$ which is a retraction to the inclusion, that is $\pi(b)=b$ for all $b\in B$.
\end{defi}

Clearly $B\subset A$ is split if and only if there exists a two-sided ideal $M$ of $A$ such that $A=B\oplus M$.

Next we provide some examples of split extensions. In the last example we add arrows to the quiver of a bound quiver algebra. Note that in relation to the finitistic dimension conjecture, E.L. Green, C. Psaroudakis and {\O}. Solberg \cite{GREENPSAROUDAKISSOLBERG} have considered the case of adding exactly one arrow, which leads to a trivial extension.

\begin{exams}\label{examples split}
\
\begin{enumerate}
\item
Let $B$ be an algebra, let $N$ be a $B$-bimodule and let $T$ be the tensor algebra $$T=T_B(N)=B\oplus\ N\ \oplus\ N\otimes_BN\ \oplus \cdots.$$  Let $T^{>i}=N^{\otimes_Bi+1}\ \oplus\ N^{\otimes_Bi+2} \oplus \cdots $

We have  $T=B\oplus T^{>0}$, that is $T$ is a split extension. Moreover, if $J\subset T^{>0}$ is a two-sided ideal of $T$, then $B\subset T/J$ is a split extension as well.
\item \label{quiver} Let $Q$ be a finite quiver, that is $Q=(Q_0,Q_1,s,t)$ where $Q_0$ and $Q_1$ are finite sets called respectively vertices and arrows, and $s,t:Q_0\to Q_1$ are maps called respectively source and target. Let $A=kQ/I$ be a bound quiver algebra, where $kQ$ is the path algebra of $Q$ and $I$ is an admissible two-sided ideal of $kQ$, see \cite{GABRIEL1972,GABRIEL1973,GABRIEL1980} and \cite{ASSEMSIMSONSKOWRONSKY,SCHIFFLER}. The extension $B=kQ_0\subset A$ is  split.

\item \label{all} Let $B=kQ/I$ be a bound quiver algebra, and let $F$ be a finite set of new arrows, that is $F$ is a finite set with two maps
$s,t:F\to Q_0$. Let $Q_F$ be the quiver with the same vertices than $Q$, while its arrows are $Q_1\sqcup F$.

         Let $B_F= kQ_F/\langle I\rangle_{kQ_F}$, where the denominator is the two-sided ideal of $kQ_F$ generated by $I$. It is easily proven that  $B_F= T_B(N)$ where
         \begin{equation}\label{N F}
         N=\bigoplus_{a\in F} Bt(a)\otimes s(a)B.
          \end{equation}

Let $J\subset B_F^{>0}$ be a two-sided ideal of $B_F$. The algebra  $$A=B_F/J = B\oplus (B_F^{>0}/J)$$ is also a split extension.

\end{enumerate}
\end{exams}

The first item of the next result is a generalisation of a reduced bar resolution obtained in \cite[Lemma 2.1]{CIBILS1990}.

\begin{theo}
Let $A=B\oplus M$ be a split extension of algebras.
\begin{enumerate}
\item There is a \emph{reduced relative bar resolution} of $A$ as $A$-bimodule
\begin{equation}\label{rrbr}\cdots\stackrel{d}{\to}A\otimes_BM^{\otimes_Bn}\otimes_BA\stackrel{d}{\to}\cdots\stackrel{d}{\to}A\otimes_BM\otimes_BA\stackrel{d}{\to}A\otimes_BA
\stackrel{d}{\to}A\to 0\end{equation}
where the formulas for the $d$'s are those of the ordinary bar resolution, see \cite{HOCHSCHILD1945,HOCHSCHILD1956}.
\suspend{enumerate}
In what follows the formulas for the (co)boundaries are the ordinary ones.
\resume{enumerate}
\item Let $X$ be an $A$-bimodule. The homology of the following chain complex $C_*^M(A|B,X)$ is $H_*(A|B,X)$.
\begin{equation}{ C_*^M(A|B,X) }: \ \  \cdots  \stackrel{b}{\to} X\otimes_{B^e} M^{\otimes_Bn}\stackrel{b}{\to}\cdots \stackrel{b}{\to} X\otimes_{B^e}M \stackrel{b}{\to} X_B\to 0\end{equation}
\item The cohomology of the following cochain complex $C_M^*(A|B,X)$ is \\$H^*(A|B,X)$.
\begin{equation}0\to X^B\stackrel{b}{\to}\Hom_{B^e}(M,X)\stackrel{b}{\to}\cdots \stackrel{b}{\to}\Hom_{B^e}(M^{\otimes_Bn},X)\stackrel{b}{\to}\cdots\end{equation}

\end{enumerate}
\end{theo}

\begin{proof}
By construction, the bimodules involved in the first item are induced bimodules, hence they are relative projective. Let $a=a_B+a_M$ be the decomposition of  $a\in A=B\oplus M$, and let
$$t(a_1\otimes m_2\otimes\cdots\otimes m_{n+1}\otimes a_{n+2})=1\otimes (a_1)_M\otimes m_2\otimes\cdots\otimes m_{n+1}\otimes a_{n+2}.$$
It is easily proven that $t$ is a well defined $B^e$-morphism, which is a contracting homotopy.

The second item is obtained by applying the functor $X\otimes_{A^e}-$ to the resolution, and the following canonical isomorphism where $Z$ is a  $B$-bimodule
$$X\otimes_{A^e}\left(A\otimes_BZ\otimes_BA\right)= X\otimes_{B^e}Z.$$
The last item is obtained analogously.\qed
\end{proof}

\begin{rema}\label{contraction homotopy A}
For later use, we record that the contracting homotopy $t$ in the previous proof is also a right $A$-module map.
\end{rema}

A $B$-bimodule $M$ is \emph{$B$-tensor nilpotent} if there exists $n$ such that $M^{\otimes_B n}=0$.  Moreover $n$ is the index of $B$-tensor nilpotency if $M^{\otimes_B n-1}\neq 0$. For instance, let $kQ$ be the path algebra of a quiver $Q$.  The $kQ_0$-bimodule $\langle Q_1\rangle\subset kQ$ is $kQ_0$-tensor nilpotent if and only if there is no oriented cycle  in $Q$.

\begin{coro}\label{tensor nilpotent H vanish in large degrees}
Let $A=B\oplus M$ be a split extension, where $M$ is $B$-tensor nilpotent of index $n$. Let $X$ be an $A$-bimodule. For $*\geq n$ we have
$$H_*(A|B,X)=0\mbox{\ \ and\ \ }H^*(A|B,X) =0.$$
\end{coro}

Let $C_*(A,X)$ be the ordinary chain complex
\begin{equation} C_*(A,X):  \cdots \stackrel{b}{\to} X\otimes A^{\otimes n}\stackrel{b}{\to}\cdots \stackrel{b}{\to} X\otimes A \stackrel{b}{\to} X\to 0\end{equation}
whose homology is the Hochschild homology $H_*(A,X)$ of an $A$-bimodule $X$.
Towards obtaining a Jacobi-Zariski long exact sequence for a split extension algebra, we observe the following.

\begin{prop}\label{sequence of complexes}
Let $A=B\oplus M$ be a split extension of algebras, and let $X$ be an $A$-bimodule. For $*\geq1$, there is a sequence of chain complexes
\begin{equation}\label{the sequence}
0\to C_*(B,X)\stackrel{\iota}{\to}  C_*(A,X)\stackrel{\kappa}{\to} C_*^M(A|B,X)\to 0
\end{equation}
where $\iota$ is injective,  $\kappa$ is surjective and $\kappa\iota=0$.

In degree $0$ we have the sequence
$$0\to X \stackrel{1}{\to} X\to X_B\to 0.$$
\end{prop}
\begin{proof}
The definition of the map $\iota$ is clear, and it is also clear that $\iota$ is an injective map of complexes. The map $\kappa$  given by
  $$x\otimes a_1\otimes\cdots\otimes a_n \mapsto x\otimes_{B^e}\left[ \left(a_1\right)_M\otimes_B\cdots\otimes_B \left(a_n\right)_M\right]$$
 is surjective, and $\kappa\iota=0$.
The verification that $\kappa$ is a map of complexes does not raise any difficulty. It uses extensively that $(aa')_M= a_Ma'_M+a_Ba'_M+a_Ma'_B$ for $a,a'\in A$ and that the first tensor product in $C_*^M(A|B,X)$ is over $B^e$.\qed
\end{proof}

\begin{rema}
Considering $C_*(A|B,X)$ instead of $C_*^M(A|B,X)$, and $\kappa'$ given by
  $$x\otimes a_1\otimes\cdots\otimes a_n \mapsto x\otimes_{B^e}\left[ a_1\otimes_B\cdots\otimes_B a_n\right]$$
does not give in general $\kappa'\iota=0$.
\end{rema}

Let $A=B\oplus M$ be a split extension. In the ensuing decomposition of the vector space $A^{\otimes n}$, let $[M_pB_q]$ be the direct sum of the direct summands containing $p$ tensorands in $M$ and $q$ tensorands in $B$, with $p+q=n$. For instance
-- omitting the $\otimes$ signs -- we have that
$$[M_2B_2]= MMBB \oplus MBMB \oplus MBBM \oplus BMBM \oplus BMMB\oplus BBMM$$
which is a direct summand of $A^{\otimes 4}$.

We set
\begin{equation}\label{K}
K_{n,0} = \Ker (X\otimes M^{\otimes n}\twoheadrightarrow X\otimes_{B^e} M^{\otimes_B n}).
\end{equation}

\begin{lemm}\label{defect}
In the situation of Proposition \ref{sequence of complexes},
$$
\arraycolsep=1mm\def\arraystretch{2}
\begin{array}{llll}
\Ker \kappa = \bigoplus_{\substack{p+q=n\\p\geq 0 \ q>0}}\ X\otimes [M_pB_q] \ \oplus\ K_{n,0}\\
\Im \iota = X\otimes [M_0B_n]\\
\Ker\kappa/\Im \iota=  \bigoplus_{\substack{p+q=n\\p>0\ q>0}}\ X\otimes [M_pB_q] \ \oplus\ K_{n,0}.
\end{array}
$$
\end{lemm}

\begin{proof}
Consider the direct sum decomposition $$X\otimes A^{\otimes n} = \bigoplus_{\substack{p+q=n\\p\geq 0\ q\geq 0}}\ X\otimes [M_pB_q].$$
If $q>0$, then $\kappa\left(X\otimes [M_pB_q]\right)=0$, hence $$\bigoplus_{\substack{p+q=n\\p\geq 0\ q>0}}\ X\otimes [M_pB_q]\subset \Ker \kappa.$$
Instead if $q=0$, then $\kappa_{\big |{X\otimes M^{\otimes n}}}$ is not zero in general and its kernel is denoted $K_{n,0}$. It follows that $\Ker \kappa$ is as stated.

In turn, in the above direct sum decomposition of $X\otimes A^{\otimes n}$, the direct summand for $p=0$ and $q=n$ is clearly $\Im \iota$. This vector space is one of the direct summands obtained above for $\Ker \kappa$. The decomposition of $\Ker\kappa/\Im \iota$ follows.
\qed
\end{proof}

\section{\sf Nearly exact sequences and the Jacobi-Zariski long exact sequence}\label{nearly}

In this section we will prove that if a sequence as (\ref{the sequence}) has zero homology for large enough degrees at the second page
of the associated spectral sequence, then there is a long exact sequence in homology starting at this precise degree.

\begin{defi}\label{sequence nearly exact}
A sequence of chain complexes concentrated in non negative degrees
$$0\to C_*\stackrel{\iota}{\to}  D_*\stackrel{\kappa}{\to} E_*\to 0$$
is  \emph{$m$-nearly exact} if
\begin{itemize}
\item[-]$\iota$ is injective,
 \item[-] $\kappa$ is surjective,
 \item[-]$\kappa\iota=0$,
 \item[-] the chain complex $\Ker \kappa/\Im \iota$ with boundary induced by the boundary of $D$, is exact in degrees  $\geq m$.
 \end{itemize}

\end{defi}

We will prove later on that under some hypotheses, the sequence of Proposition \ref{sequence of complexes} is nearly exact.

\begin{theo}\label{long exact sequence}
Let
\begin{equation}\label{nearly exact}
0\to C_*\stackrel{\iota}{\to}  D_*\stackrel{\kappa}{\to} E_*\to 0
\end{equation}
be a  $m$-nearly exact sequence of chain complexes. There is a long exact sequence as follows:
\begin{equation*}
\dots \stackrel{\delta}{\to} H_{m+1}(C) \stackrel{\iota}{\to} H_{m+1}(D) \stackrel{\kappa}{\to} H_{m+1}(E) \stackrel{\delta}{\to} H_{m}(C)  \stackrel{\iota}{\to} H_{m}(D).
\end{equation*}
\end{theo}

\begin{proof}
We will use standard results on spectral sequences, see for instance  \cite{MCCLEARY} or \cite{WEIBEL}.

The homological double complex arising from the sequence (\ref{nearly exact})  with the standard change of signs, has the complexes $E$, $D$ and $C$ at columns $p=0$, $1$ and $2$ respectively. Firstly we claim that this complex has zero homology in  total degrees $\geq m+1$. Indeed, consider the spectral sequence given by the filtration by the rows. At the first page the columns corresponding to $p=0,2$  are zero since $\iota$ is injective and $\kappa$
is surjective. At column $p=1$ we have the homology vector spaces of the sequence (\ref{nearly exact}) corresponding to the complex
in the middle. Since the sequence is $m$-nearly exact, at the second page the column $p=1$ has zeros in degrees $\geq m$, and
 zeros elsewhere. This proves the claim.

Secondly we consider the filtration by columns. In page $1$ of the corresponding spectral sequence,  let $\iota_1$ and $\kappa_1$ be the horizontal maps induced by $\iota$ and $\kappa$ at the homology level of the complexes of the sequence (\ref{nearly exact}). They are the morphisms of the intended long exact sequence. We assert that in degrees $\geq m+1$ there is exactness at
the column corresponding to the homology of $D$. Indeed, the vector spaces at the second page at column $p=1$ are
$\Ker \kappa_1/\Im \iota_1$. At these spots the differentials $d_2$ come from zero and go to zero. Hence these vector spaces live forever in the subsequent pages of the spectral sequence. We proved before that the complex has no homology in total  degrees $\geq m+1$, hence these vector spaces vanish in degrees $\geq m+1$.

Finally we turn to the connecting homomorphism $\delta$. In the second page of the spectral sequence just considered, at columns $p=0,2$ we have respectively $\Coker\kappa_1$ and $\Ker\iota_1$. We assert that the differentials  $d_2: \Ker\iota_1 \to \Coker \kappa_1$  from total degree $m+1$ to total degree $m$ are isomorphisms, as well as in larger degrees.  Indeed,  in these degrees $\Ker d_2$ and $\Coker d_2$ live forever in the spectral sequence, hence they vanish by the same argument than above. We assert that composing $d_2^{-1}$  with the inclusion of $\Ker\iota_1$ and the canonical projection to $\Coker\kappa_1$ provides the required connecting homomorphism $\delta$ of the long exact sequence for these degrees. Indeed, by construction $\Ker\delta=\Im \kappa_1$ and $\Im \delta= \Ker\iota_1$. \qed

\end{proof}

For the next result we assume that the $B$-bimodule $M$ verifies that for $*>0$ we have $\Tor_*^B(M, M^{\otimes_B n})=0$ for all $n$. Note that this is fulfilled if $M$ is either a left or a right projective $B$-module.

\normalsize

\begin{prop}\label{F}
Let $A=B\oplus M$ be a split algebra, let $X$  be an $A$-bimodule and consider the sequence (\ref{the sequence}) as a double complex after performing the standard change of signs. Let $E^2_{1,*}$ be the second page
of the spectral sequence obtained by filtering by rows.

There is a double complex $C_{*,*}$ which total homology is $E^2_{1,*}$. The filtration by columns of $C_{*,*}$  yields a spectral sequence.  If $M$ verifies that for $*>0$ we have $\Tor_*^B(M, M^{\otimes_B n})=0$ for all $n$, the terms at page $1$ are
$$F^1_{p,q}= \Tor^{B^e}_{p+q}(X,M^{\otimes_Bp}) \mbox{ for } p,q>0$$
and $0$ otherwise.

\end{prop}

\begin{proof}
By Lemma \ref{defect},
 $$E^1_{1,n}=\bigoplus_{\substack{p+q=n\\p>0\ q>0}}\ X\otimes [M_pB_q] \ \oplus\ K_{n,0}.$$ The differential of this column is deduced from the one of $C_*(A,X)$. Clearly this column is the total complex of the double chain complex.
\begin{itemize}
\item $C_{p,q}=X\otimes [M_pB_q]$ for $p,q>0$,
\item $C_{p,0}=K_{p,0}$ for $p>0,$
\item  $0$ at other spots.
\end{itemize}
 We  modify momentarily $C_{*,*}$  at its bottom line  as follows: \\
 \centerline{$C'_{*,*}=  X\otimes [M_pB_q]$ for $p>0, q\geq 0$, and $0$ at other spots,}\\
 with differentials still inherited from $C_*(A,X)$.

We assert that the homology of the column  $p=1$ of $C'_{*,*}$ is $\Tor^{B^e}_*(X,M)$. To this purpose, we next recall a specific projective resolution of $M$ as a $B^e$-module which is provided in the proof of Proposition 4.1 of \cite{CIBILS2000}.  We have that the functor $X\otimes_{B^e}-$ applied to it yields the mentioned column, proving this way the assertion. As before, we omit the tensor product sign $\otimes$ over $k$.

 Let
$${}^qM^p=\underbrace{B\cdots B}_q M\underbrace{ B\cdots  B}_p$$
and consider the following complex of free $B^e$-modules:
$$\cdots\stackrel{d}{\to}\displaystyle \bigoplus_{\substack{ p+q=n+1\\ p>0\ q>0}} {}^qM^p\stackrel{d}{\to}\cdots\stackrel{d}{\to}\  {}^1M^2\oplus {}^2M^1 \ \stackrel{d}{\to}\  {}^1M^1 \ \stackrel{d}{\to} M\to 0,$$
where the first differential is special, namely $d(b\otimes m\otimes b')=bmb'$. In larger degrees, the differential is the differential of the total complex of the double complex which has ${}^qM^p$ at the spot $(q,p)$, with vertical and horizontal differentials  ${}^qM^p \to {}^qM^{p-1}$   and ${}^qM^p \to {}^{q-1}M^{p}$ given respectively by
\vskip1mm
$b_1\otimes \cdots\otimes b_q\otimes m \otimes b'_1\otimes\cdots\otimes b'_p\mapsto \\ (-1)^{q+1}[b_1\otimes \cdots\otimes b_q\otimes m b'_1\otimes\cdots\otimes b'_p+\\
\sum (-1)^i b_1\otimes \cdots\otimes b_q\otimes m \otimes b'_1\otimes\cdots \otimes b'_ib_{i+1} \otimes\cdots\otimes b'_p]$
\vskip1mm
and
\vskip1mm
$b_1\otimes \cdots\otimes b_q\otimes m \otimes b'_1\otimes\cdots\otimes b'_p\mapsto \\
\sum (-1)^i b_1\otimes \cdots \otimes b_ib_{i+1} \otimes \cdots \otimes b_q\otimes m \otimes b'_1\otimes\cdots\otimes b'_p+\\
(-1)^q b_1\otimes \cdots\otimes b_q m \otimes b'_1\otimes\cdots\otimes b'_p.$

\vskip1mm

We make precise that the vertical and horizontal differentials  ${}^qM^1 \to {}^qM^{0}$   and ${}^1M^p \to {}^{0}M^{p}$ are given respectively by
\vskip1mm
$b_1\otimes \cdots\otimes b_q\otimes m \otimes b'_1\mapsto (-1)^{q+1}b_1\otimes \cdots\otimes b_q\otimes m b'_1$
\vskip1mm
and
\vskip1mm
$b_1\otimes m \otimes b'_1\otimes\cdots\otimes b'_p\mapsto
- b_1m \otimes b'_1\otimes\cdots\otimes b'_p.$

\vskip1mm
\normalsize

For the column $p=2$  of $C'_{*,*}$, let $F_\bullet \to M$ be the previous projective resolution  of $M$ by free bimodules. Its tensor product over $B$ with the left bar resolution
$$\cdots \to BBM\to BM\to M\to 0$$ of $M$ provides the following double complex $\mathbf D$:
\vskip3mm

\begin{tikzcd}
  & \vdots \arrow[d]                  & \vdots \arrow[d]                    & \vdots \arrow[d]                    &                  \\
0 & BBM\otimes_BM \arrow[d] \arrow[l] & BBM\otimes_BF_0 \arrow[d] \arrow[l] & BBM\otimes_BF_1 \arrow[d] \arrow[l] & \cdots \arrow[l] \\
0 & BM\otimes_BM \arrow[d] \arrow[l]  & BM\otimes_BF_0 \arrow[d] \arrow[l]  & BM\otimes_BF_1 \arrow[d] \arrow[l]  & \cdots \arrow[l] \\
0 & M\otimes_BM \arrow[l] \arrow[d]   & M\otimes_BF_0 \arrow[l] \arrow[d]   & M\otimes_BF_1 \arrow[l] \arrow[d]   & \cdots \arrow[l] \\
  & 0                                 & 0                                   & 0                                   &
\end{tikzcd}

\vskip3mm
The bar resolution has a right $B$-module contracting homotopy. Hence the columns of $\mathbf D$ are acyclic and its total complex $\mathsf{Tot}(\mathbf{D})$ is exact.

However the bimodules of the bottom row and of the left column are not projective in general, while the others are. In order to obtain the required projective resolution of $M\otimes_B M$ as a bimodule, we proceed as in \cite{CIBILS2000}. Let $\mathbf S$ be the double subcomplex of  $\mathbf D$ given by the bottom row and the left column. We claim that $\mathsf{Tot}(\mathbf{D/S}) \to M\otimes_BM$ is a free resolution of $M\otimes_BM$ as a $B$-bimodule.

First note that the homology of the bottom row is precisely $\Tor_*^B(M,M)$, which is zero in positive degrees by  hypothesis. Hence $\mathsf{Tot}(\mathbf S)$ is exact in positive degrees, while in degree zero its homology is $M\otimes_B M$: indeed, observe that for surjective morphisms $f: Y\to X$ and $g:Z\to X$, and $(f,g): Y\oplus Z\to X$ we have that
$$\frac{\ker (f,g)}{\ker f\oplus\ker g} \mbox{ is isomorphic to } X$$
by the map induced by $f$ or by $g$.

Next we consider the long exact sequence associated to the exact sequence of complexes
 $$0\to \mathbf S \to \mathbf D\to \mathbf {D/S}\to 0.$$

 It shows that  $\mathsf{Tot}(\mathbf{D/S})$ is acyclic except in its last term where the homology is $M\otimes_B M$. This provides the required resolution of  $M\otimes_B M$.

 Next we iterate the process by tensoring the last resolution with the left bar resolution of $M$. As before, we use that  $\Tor_*^B(M,M\otimes_B M)=0$ in positive degrees to infer a projective resolution of the $B$-bimodule $M\otimes_B M\otimes_B M$.

This shows that the homology of the $p$-th column is   $\Tor^{B^e}_{*}(X,M^{\otimes_Bp})$.

\normalsize

In order to return to $C_{*,*}$, note that by (\ref{K}) we have $$(X\otimes M^{\otimes p})/K_{p,0} = X\otimes_{B^e} M^{\otimes_B p} = \Tor^{B^e}_0(X,M^{\otimes_Bp}).$$
Hence replacing the bottom row of $C'$ by  $K_{*,0}$ yields surjective maps at the bottom stage of each column, therefore we have zero homology at spots of the bottom row of $C$. \qed
\end{proof}

Next we provide sufficient conditions to ensure that the sequence (\ref{the sequence}) of Proposition \ref{sequence of complexes} is nearly exact.

\begin{defi}\label{bounded}
Let $B$ be an algebra. A $B$-bimodule $M$ is \emph{bounded} if
\begin{itemize}
\item[-] $M$ is $B$-tensor nilpotent,
\item [-] $M$ is of finite projective dimension as a $B^e$-module,
\item [-] $M$ is either a left or a right projective $B$-module.
\end{itemize}
 \end{defi}

 \begin{rema}
 Let $B$ be an algebra, and let $M$ be a $B$-bimodule with a \emph{$B$-associative structure,} that is an associative map of $B$-bimodules $M\otimes_B M\to M$. Then $B\oplus M$ is a split extension algebra. Of course all split extensions occurs this way.

We underline that in the requirement that $M$ is bounded, the $B$-associative structure of $M$ is not involved.    \end{rema}

\begin{defi} \label{split bounded extension} A \emph{ split bounded extension} $B\oplus M$ is a split extension where $M$ is bounded.
\end{defi}
\begin{prop}\label{it is nearly exact}
 Let $A=B\oplus M$ be  a split bounded extension. Let $n$ be the index of $B$-tensor nilpotency of $M$. Let $u$ be the projective dimension of $M$ as a $B^e$-module, and let $X$ be an $A$-bimodule.

The sequence (\ref{the sequence}) is $nu$-nearly exact.
\end{prop}
\begin{proof}
We consider the spaces $F_{p,q}^1= \Tor^{B^e}_{p+q}(M^{\otimes_Bp}, X) \mbox{ for } p,q>0$ of Proposition \ref{F}. On the one hand $F_{p,q}^1=0$ for $p\geq v$.

On the other hand, from \cite[Chapter IX, Proposition 2.6]{CARTANEILENBERG} we infer that since $M$ is projective either as left or as right
$B$-module, and is of projective dimension $u$ as a $B$-bimodule, then $M^{\otimes_Bp}$ is of projective dimension  $\leq$
$pu$. Hence if $p+q\geq pu$, then $F_{p,q}^1=0$.

As a consequence, if $p+q\geq nu$, then $F_{p,q}^1=0$. By Proposition \ref{F} we obtain that if $*\geq nu$ then $E^2_{1,*}=0$, which means that the column of homologies from the middle of the sequence  (\ref{the sequence}) has in turn no homology in degrees $\geq nu$, that is the sequence is $nu$-nearly exact.\qed

\end{proof}

The previous result and Theorem \ref{long exact sequence} prove the following.
\begin{theo}\label{JZ nosotros}
Let $A=B\oplus M$ be a split bounded extension  as in Proposition \ref{it is nearly exact}, and let $X$ be an $A$-bimodule. There is a Jacobi-Zariski long exact sequence as follows.
\begin{equation*}
\begin{split}
\dots \stackrel{\delta}{\to} & H_{nu+1}m(B,X) \stackrel{\iota}{\to} H_{nu+1}(A,X) \stackrel{\kappa}{\to} H_{nu+1}(A|B,X) \stackrel{\delta}{\to} \\& H_{nu}(B,X)\stackrel{\iota}{\to} H_{nu}(A,X).
\end{split}
\end{equation*}

\end{theo}

\section{\sf Han's conjecture}\label{Han}

A finite dimensional algebra is called  \emph{smooth} if it is of finite global dimension. As it is mentioned in the Introduction,
the word smooth is originated in commutative algebra and is useful here for brevity.  Note that for noetherian rings, the left and right global dimensions are equal, see \cite{AUSLANDER}.

Han's conjecture states that for $A$ a finite dimensional algebra,  $H_*(A,A)$ vanishes in large enough degrees if and only if $A$ is smooth. Let $\mathcal{H}$ be the class of finite dimensional algebras which verify Han's conjecture. Our aim is to prove the following.
\begin{theo}
Let $A=B\oplus M$ be a split bounded extension of finite dimensional algebras.

$A\in\mathcal{H}$ if and only if  $B\in\mathcal{H}.$
\end{theo}
The proof relies on the next result.
\begin{prop}
Let $A=B\oplus M$ be a split bounded extension of finite dimensional algebras.
\begin{enumerate}
\item \label{1} $H_*(A,A)$ vanishes in large enough degrees if and only if $H_*(B,B)$ vanishes in large enough degrees.
\item  $A$ is smooth if and only if $B$ is smooth.
\end{enumerate}
\end{prop}

\begin{rema} For a split extension $A=B\oplus M$, it is trivial that if $H_*(A,A)$ vanishes in large enough degrees, then the same happens
for $H_*(B,B)$. Indeed,  $H_*(-,-)$ is a functor on the category of algebras. Hence $H_*(B,B)$ is a direct summand of $H_*(A,A)$.
\end{rema}

\begin{proof}
\begin{enumerate}
\item

Recall that $n$ is the index of $B$-nilpotency of $M$, and $u$ is the projective dimension of $M$ as a $B^e$-module. We claim that $H_*(A,A)$ and $H_*(B,B)$ are isomorphic if $*\geq nu+1$.

     Recall that by Corollary \ref{tensor nilpotent H vanish in large degrees} we have that $H_*(A|B,A)$ vanishes for $*\geq n$. Hence  the Jacobi-Zariski long exact sequence of Theorem \ref{JZ nosotros} shows that $H_*(B,A)$ and $H_*(A,A)$ are isomorphic for $*\geq nu+1$.

     On the other hand
    $H_*(B,A)= H_*(B,B)\oplus H_*(B,M)$. Moreover, $$H_*(B,M)=\Tor_*^{B^e}(B,M).$$
         Hence if $*\geq u$, then $H_*(B,M)=0$ and $H_*(B,A)= H_*(B,B)$.

\item

The bimodule $M$ is projective from at least  one side, we will suppose that $M$ is right projective.
First we prove that if $A$ is smooth then $B$ is smooth. Let $\pi:A\to B$ be the retraction algebra map of $B\subset A$, with kernel $M$. Let $Y$ be a right $B$-module, and let $\overline{Y}$ be the $A$-module obtained by restricting scalars through $\pi$. We have $Y=\overline{Y}$ as right $B$-modules and $\overline{Y}M=0$.

Let $P_*\to \overline{Y}$ be a finite right $A$-projective resolution of $\overline{Y}$. It remains of course exact when considering it as an exact sequence of right $B$-modules. Moreover, if $P$ is a right projective $A$-module then  it is also projective as a right $B$-module. Indeed, this is true for the free rank one $A$-module $B\oplus M$. Then the standard arguments enable to conclude.

To prove that if $B$ is smooth then so is $A$, we begin by proving that any induced $A$-module $Z=A\otimes_BY$ is of finite projective dimension.  Let $Q_*\to Y$ be a finite left $B$-projective resolution of $Y$. The functor $A\otimes_B-$ is exact since $A$ is right projective. Moreover if $Q$ is a left projective $B$-module, then $A\otimes_BQ$ is a left projective $A$-module,  this follows from the fact that
this is true for $Q=B$. Therefore $A\otimes_BQ_*\to A\otimes_BY$ is a finite left $A$-projective resolution of the induced module $Z$.

Let $X$ be a left $A$-module. We claim that there exists an exact sequence of $A$-modules
$0\to Z_n\to Z_{n-1}\to\cdots\to Z_0\to X\to 0$
where the $A$-modules $Z_i$ are induced. This claim ends the proof, indeed each $Z_i$ is of finite projective dimension by the previous assertion, hence $X$ is of finite projective dimension.

To prove the claim, consider the relative reduced bar resolution \ref{rrbr}, which is finite since $M$ is $B$-tensor nilpotent. Moreover
its contracting homotopy is a right $A$-module map, see Remark \ref{contraction homotopy A}. Consequently the relative reduced bar resolution remains exact by applying the functor $-\otimes_AX$. For some $n$ we obtain
$$0{\to}A\otimes_BM^{\otimes_Bn}\otimes_BX{\to}\cdots{\to}A\otimes_BM\otimes_BX{\to}A\otimes_BX
{\to} X\to 0.$$
Note that all the $A$-modules except $X$ are induced $A$-modules.\qed

    \end{enumerate}
\end{proof}

\footnotesize
\noindent C.C.:\\
Institut Montpelli\'{e}rain Alexander Grothendieck, CNRS, Univ. Montpellier, France.\\
{\tt Claude.Cibils@umontpellier.fr}

\medskip
\noindent M.L.:\\
Instituto de Matem\'atica y Estad\'\i stica  ``Rafael Laguardia'', Facultad de Ingenier\'\i a, Universidad de la Rep\'ublica, Uruguay.\\
{\tt marclan@fing.edu.uy}

\medskip
\noindent E.N.M.:\\
Departamento de Matem\'atica, IME-USP, Universidade de S\~ao Paulo, Brazil.\\
{\tt enmarcos@ime.usp.br}

\medskip
\noindent A.S.:
\\IMAS-CONICET y Departamento de Matem\'atica,
 Facultad de Ciencias Exactas y Naturales,\\
 Universidad de Buenos Aires, Argentina. \\{\tt asolotar@dm.uba.ar}

\end{document}